\documentclass[11pt]{article}
\usepackage{latexsym,amssymb,amsmath,amsthm,enumerate,geometry,cite,float,tikz}
\newtheorem{theorem}{Theorem}
\newtheorem{lemma}[theorem]{Lemma}

\newtheorem{claim}{Claim}

\usepackage{lineno}
\usepackage{setspace}
\allowdisplaybreaks

\begin{document}
\onehalfspace

\title{Diameter, edge-connectivity, and $C_4$-freeness}
\author{Vanessa Hiebeler\and Johannes Pardey\and Dieter Rautenbach}
\date{}

\maketitle
\vspace{-10mm}
\begin{center}
{\small 
Institute of Optimization and Operations Research, Ulm University, Ulm, Germany\\
\texttt{$\{$vanessa.hiebeler,johannes.pardey,dieter.rautenbach$\}$@uni-ulm.de}
}
\end{center}

\begin{abstract}
Improving a recent result of Fundikwa, Mazorodze, and Mukwembi,
we show that $d \leq (2n-3)/5$ 
for every connected $C_4$-free graph of 
order $n$, 
diameter $d$, 
and edge-connectivity at least $3$,
which is best possible up to a small additive constant.
For edge-connectivity at least $4$, we improve this to $d \leq (n-3)/3$.
Furthermore, adapting a construction due to Erd\H{o}s, Pach, Pollack, and Tuza, 
for an odd prime power $q$ at least $7$, and every positive integer $k$, 
we show the existence of a 
connected $C_4$-free graph of 
order $n=(q^2+q-1)k+1$, 
diameter $d=4k$, 
and edge-connectivity $\lambda$ at least $q-6$, in particular,
$d\geq 4(n-1)/(\lambda^2+O(\lambda))$.\\[3mm]
\end{abstract}

\section{Introduction}

All graphs considered here are finite, simple, and undirected.
The fact that a connected graph 
of minimum degree $\delta$ and diameter $d$ 
has order at least $\frac{d(\delta+1)}{3}+O(1)$
has been discovered several times \cite{amfoge,erpapotu,gomafa,mo}.
Lower bounds on the order of graphs realizing a certain diameter
subject to conditions such as
triangle-freeness \cite{erpapotu,fumamu2},
$C_4$-freeness \cite{erpapotu,fumamu}, or 
conditions involving the chromatic number \cite{czsisz,czsisz2}
or the edge-connectivity have been studied \cite{almamuve,fumamu,fumamu2}.
Similar problems for notions related to the diameter also received attention
\cite{dajoma,daosmuro,fumamu3}.

In the present paper we relate the order and the diameter 
for $C_4$-free graphs subject to edge-connectivity conditions.
Let $G$ be a connected graph of 
order $n$, 
minimum degree $\delta\geq 2$, 
diameter $d$, 
and edge-connectivity $\lambda$ that is $C_4$-free,
that is, the graph $G$ contains no cycle of length four as a 
(not necessarily induced) subgraph.
Exploiting the simple fact 
that there are at least $\delta^2-2\left\lfloor\frac{\delta}{2}\right\rfloor+1$
vertices within distance at most $2$ from every vertex of $G$,
Erd\H{o}s et al.~\cite{erpapotu} showed
\begin{eqnarray}\label{e1}
d & \leq & \frac{5n}{\delta^2-2\left\lfloor\frac{\delta}{2}\right\rfloor+1}.
\end{eqnarray}
Furthermore, they showed that, provided that $q=\delta+1$ is a prime power,
there is a graph $G$,
derived from the so-called {\it Brown graph $B(q)$} \cite{br,erre,basi},
with the above properties that satisfies
\begin{eqnarray}\label{e1b}
d & \geq & \frac{5n}{\delta^2+3\delta+2}-1,
\end{eqnarray}
that is, asymptotically, there is only little room for improvements of (\ref{e1}).
If $\lambda\geq 3$, then $\delta\geq 3$, and (\ref{e1}) implies $d\leq \frac{5n}{8}$,
which was recently improved by Fundikwa et al.~\cite{fumamu} to
\begin{eqnarray}\label{e2}
d & \leq & \frac{3n-3}{7}.
\end{eqnarray}
They also constructed graphs with 
$d \geq \frac{3n-5}{8}$
for infinitely many $n$,
that is, the factor $3/7$ in (\ref{e2}) can not be improved beyond $3/8$.

Our first result is the following improvement of (\ref{e2}).

\begin{theorem}\label{theorem1}
If $G$ is a connected $C_4$-free graph of 
order $n$, 
diameter $d$, 
and edge-connectivity at least $3$,
then
\begin{eqnarray}\label{e3}
d & \leq & \frac{2n-3}{5}.
\end{eqnarray}
\end{theorem}
The graphs illustrated in Figure \ref{fig1} show that the factor $2/5$ in (\ref{e3})
is best possible, that is, our result is tight up to a small additive constant. 
\begin{figure}[H]
\begin{center}
\unitlength 0.9mm 
\linethickness{0.4pt}
\ifx\plotpoint\undefined\newsavebox{\plotpoint}\fi 
\begin{picture}(186,36)(0,0)
\put(45,10){\circle*{1}}
\put(65,10){\circle*{1}}
\put(85,10){\circle*{1}}
\put(105,10){\circle*{1}}
\put(125,10){\circle*{1}}
\put(145,10){\circle*{1}}
\put(45,20){\circle*{1}}
\put(65,20){\circle*{1}}
\put(85,20){\circle*{1}}
\put(105,20){\circle*{1}}
\put(125,20){\circle*{1}}
\put(145,20){\circle*{1}}
\put(35,15){\circle*{1}}
\put(25,5){\circle*{1}}
\put(165,5){\circle*{1}}
\put(25,25){\circle*{1}}
\put(165,25){\circle*{1}}
\put(55,15){\circle*{1}}
\put(75,15){\circle*{1}}
\put(95,15){\circle*{1}}
\put(115,15){\circle*{1}}
\put(135,15){\circle*{1}}
\put(45,30){\circle*{1}}
\put(65,30){\circle*{1}}
\put(85,30){\circle*{1}}
\put(105,30){\circle*{1}}
\put(125,30){\circle*{1}}
\put(145,30){\circle*{1}}
\put(35,25){\circle*{1}}
\put(25,15){\circle*{1}}
\put(165,15){\circle*{1}}
\put(25,35){\circle*{1}}
\put(165,35){\circle*{1}}
\put(55,25){\circle*{1}}
\put(75,25){\circle*{1}}
\put(95,25){\circle*{1}}
\put(115,25){\circle*{1}}
\put(135,25){\circle*{1}}
\put(45,30){\line(-2,-1){10}}
\put(65,30){\line(-2,-1){10}}
\put(85,30){\line(-2,-1){10}}
\put(105,30){\line(-2,-1){10}}
\put(125,30){\line(-2,-1){10}}
\put(145,30){\line(-2,-1){10}}
\put(35,25){\line(2,-1){10}}
\put(55,25){\line(2,-1){10}}
\put(75,25){\line(2,-1){10}}
\put(95,25){\line(2,-1){10}}
\put(115,25){\line(2,-1){10}}
\put(135,25){\line(2,-1){10}}
\put(45,20){\line(0,1){10}}
\put(65,20){\line(0,1){10}}
\put(85,20){\line(0,1){10}}
\put(105,20){\line(0,1){10}}
\put(125,20){\line(0,1){10}}
\put(145,20){\line(0,1){10}}
\put(45,30){\line(2,-1){10}}
\put(65,30){\line(2,-1){10}}
\put(85,30){\line(2,-1){10}}
\put(105,30){\line(2,-1){10}}
\put(125,30){\line(2,-1){10}}
\put(145,30){\line(2,-1){10}}
\put(55,25){\line(-2,-3){10}}
\put(75,25){\line(-2,-3){10}}
\put(95,25){\line(-2,-3){10}}
\put(115,25){\line(-2,-3){10}}
\put(135,25){\line(-2,-3){10}}
\put(155,25){\line(-2,-3){10}}
\put(45,10){\line(2,1){10}}
\put(65,10){\line(2,1){10}}
\put(85,10){\line(2,1){10}}
\put(105,10){\line(2,1){10}}
\put(125,10){\line(2,1){10}}
\put(145,10){\line(2,1){10}}
\put(55,15){\line(-2,1){10}}
\put(75,15){\line(-2,1){10}}
\put(95,15){\line(-2,1){10}}
\put(115,15){\line(-2,1){10}}
\put(135,15){\line(-2,1){10}}
\put(155,15){\line(-2,1){10}}
\put(35,15){\line(2,-1){10}}
\put(55,15){\line(2,-1){10}}
\put(75,15){\line(2,-1){10}}
\put(95,15){\line(2,-1){10}}
\put(115,15){\line(2,-1){10}}
\put(135,15){\line(2,-1){10}}
\put(25,35){\line(1,-1){10}}
\put(165,35){\line(-1,-1){10}}
\put(35,25){\line(0,1){0}}
\put(155,25){\line(0,1){0}}
\put(35,25){\line(-1,0){10}}
\put(155,25){\line(1,0){10}}
\put(25,25){\line(0,1){10}}
\put(165,25){\line(0,1){10}}
\put(25,35){\line(-2,-1){10}}
\put(165,35){\line(2,-1){10}}
\put(15,30){\line(0,-1){10}}
\put(175,30){\line(0,-1){10}}
\put(15,20){\line(2,-1){10}}
\put(175,20){\line(-2,-1){10}}
\put(25,15){\line(1,0){10}}
\put(165,15){\line(-1,0){10}}
\put(35,15){\line(-1,-1){10}}
\put(155,15){\line(1,-1){10}}
\put(25,5){\line(0,1){10}}
\put(165,5){\line(0,1){10}}
\put(25,5){\line(-2,1){10}}
\put(165,5){\line(2,1){10}}
\put(15,10){\line(2,3){10}}
\put(175,10){\line(-2,3){10}}
\put(15,10){\line(-1,1){10}}
\put(175,10){\line(1,1){10}}
\put(5,20){\line(1,0){10}}
\put(185,20){\line(-1,0){10}}
\put(15,30){\line(-1,-1){10}}
\put(175,30){\line(1,-1){10}}
\put(15,30){\circle*{1}}
\put(175,30){\circle*{1}}
\put(15,20){\circle*{1}}
\put(175,20){\circle*{1}}
\put(15,10){\circle*{1}}
\put(175,10){\circle*{1}}
\put(5,20){\circle*{1}}
\put(185,20){\circle*{1}}
\put(155,25){\circle*{1}}
\put(155,15){\circle*{1}}
\put(1,20){\makebox(0,0)[cc]{$u$}}
\end{picture}
\end{center}
\vspace{-8mm}
\caption{A family of $C_4$-free graphs $G$ with 
$(n_0,n_1,\ldots,n_d)=(1,3,4,\overbrace{2,3,2,3,\ldots,2,3}^{\mbox{$k\times (2,3)$}},2,4,3,1)$, where $n_i$ is the number of vertices at distance $i$ from $u$.
The graph $G$, depending on the integer parameter $k$, 
is $3$-edge-connected, and satisfies $d=2k+6$ and $n=5k+18$, that is, $d=\frac{2n-6}{5}$.}\label{fig1}
\end{figure}
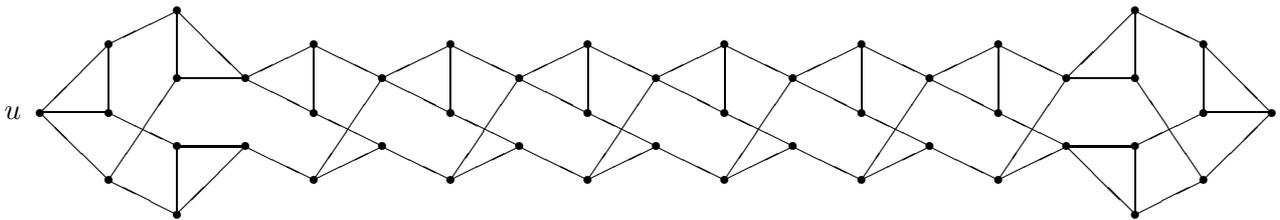
For $\delta\geq \lambda\geq 4$, the inequality (\ref{e1}) implies $d\leq \frac{5n}{13}$,
which we improve as follows.

\begin{theorem}\label{theorem1b}
If $G$ is a connected, $C_4$-free graph of order $n$, diameter $d$, 
and edge-connectivity $\lambda$ at least $4$, then
\begin{eqnarray}\label{e3b}
d \leq \frac{n-3}{3}.
\end{eqnarray}
\end{theorem}
We believe that the factor $1/3$ in (\ref{e3b}) can be improved to $2/7$,
which would be best possible in view of the graphs illustrated in Figure \ref{fig1b}.

\begin{figure}[H]
\begin{center}
\unitlength 0.5mm 
\linethickness{0.4pt}
\ifx\plotpoint\undefined\newsavebox{\plotpoint}\fi 
\begin{picture}(290,118)(0,0)
\put(46,60){\circle*{2}}
\put(86,60){\circle*{2}}
\put(126,60){\circle*{2}}
\put(166,60){\circle*{2}}
\put(206,60){\circle*{2}}
\put(246,60){\circle*{2}}
\put(46,100){\circle*{2}}
\put(86,100){\circle*{2}}
\put(126,100){\circle*{2}}
\put(166,100){\circle*{2}}
\put(206,100){\circle*{2}}
\put(246,100){\circle*{2}}
\put(46,20){\circle*{2}}
\put(86,20){\circle*{2}}
\put(126,20){\circle*{2}}
\put(166,20){\circle*{2}}
\put(206,20){\circle*{2}}
\put(246,20){\circle*{2}}
\put(66,20){\circle*{2}}
\put(106,20){\circle*{2}}
\put(146,20){\circle*{2}}
\put(186,20){\circle*{2}}
\put(226,20){\circle*{2}}
\put(66,50){\circle*{2}}
\put(106,50){\circle*{2}}
\put(146,50){\circle*{2}}
\put(186,50){\circle*{2}}
\put(226,50){\circle*{2}}
\put(66,70){\circle*{2}}
\put(106,70){\circle*{2}}
\put(146,70){\circle*{2}}
\put(186,70){\circle*{2}}
\put(226,70){\circle*{2}}
\put(66,100){\circle*{2}}
\put(106,100){\circle*{2}}
\put(146,100){\circle*{2}}
\put(186,100){\circle*{2}}
\put(226,100){\circle*{2}}
\put(66,20){\line(0,1){30}}
\put(106,20){\line(0,1){30}}
\put(146,20){\line(0,1){30}}
\put(186,20){\line(0,1){30}}
\put(226,20){\line(0,1){30}}
\put(66,70){\line(0,1){30}}
\put(106,70){\line(0,1){30}}
\put(146,70){\line(0,1){30}}
\put(186,70){\line(0,1){30}}
\put(226,70){\line(0,1){30}}
\put(46,60){\line(0,-1){40}}
\put(86,60){\line(0,-1){40}}
\put(126,60){\line(0,-1){40}}
\put(166,60){\line(0,-1){40}}
\put(206,60){\line(0,-1){40}}
\put(246,60){\line(0,-1){40}}
\put(66,100){\line(-1,0){20}}
\put(106,100){\line(-1,0){20}}
\put(146,100){\line(-1,0){20}}
\put(186,100){\line(-1,0){20}}
\put(226,100){\line(-1,0){20}}
\put(46,100){\line(2,-3){20}}
\put(86,100){\line(2,-3){20}}
\put(126,100){\line(2,-3){20}}
\put(166,100){\line(2,-3){20}}
\put(206,100){\line(2,-3){20}}
\put(66,70){\line(-2,-1){20}}
\put(106,70){\line(-2,-1){20}}
\put(146,70){\line(-2,-1){20}}
\put(186,70){\line(-2,-1){20}}
\put(226,70){\line(-2,-1){20}}
\put(66,50){\line(-2,-3){20}}
\put(106,50){\line(-2,-3){20}}
\put(146,50){\line(-2,-3){20}}
\put(186,50){\line(-2,-3){20}}
\put(226,50){\line(-2,-3){20}}
\put(46,20){\line(1,0){20}}
\put(86,20){\line(1,0){20}}
\put(126,20){\line(1,0){20}}
\put(166,20){\line(1,0){20}}
\put(206,20){\line(1,0){20}}
\put(66,100){\line(1,0){20}}
\put(106,100){\line(1,0){20}}
\put(146,100){\line(1,0){20}}
\put(186,100){\line(1,0){20}}
\put(226,100){\line(1,0){20}}
\put(86,100){\line(-2,-5){20}}
\put(126,100){\line(-2,-5){20}}
\put(166,100){\line(-2,-5){20}}
\put(206,100){\line(-2,-5){20}}
\put(246,100){\line(-2,-5){20}}
\put(86,20){\line(-1,0){20}}
\put(126,20){\line(-1,0){20}}
\put(166,20){\line(-1,0){20}}
\put(206,20){\line(-1,0){20}}
\put(246,20){\line(-1,0){20}}
\put(66,20){\line(1,2){20}}
\put(106,20){\line(1,2){20}}
\put(146,20){\line(1,2){20}}
\put(186,20){\line(1,2){20}}
\put(226,20){\line(1,2){20}}
\put(86,60){\line(-2,1){20}}
\put(126,60){\line(-2,1){20}}
\put(166,60){\line(-2,1){20}}
\put(206,60){\line(-2,1){20}}
\put(246,60){\line(-2,1){20}}
\qbezier(66,100)(76,75)(66,50)
\qbezier(106,100)(116,75)(106,50)
\qbezier(146,100)(156,75)(146,50)
\qbezier(186,100)(196,75)(186,50)
\qbezier(226,100)(236,75)(226,50)
\put(26,10){\circle*{2}}
\put(266,10){\circle*{2}}
\put(26,50){\circle*{2}}
\put(266,50){\circle*{2}}
\put(26,90){\circle*{2}}
\put(266,90){\circle*{2}}
\put(26,30){\circle*{2}}
\put(266,30){\circle*{2}}
\put(26,70){\circle*{2}}
\put(266,70){\circle*{2}}
\put(26,110){\circle*{2}}
\put(266,110){\circle*{2}}
\put(26,10){\line(2,1){20}}
\put(266,10){\line(-2,1){20}}
\put(46,20){\line(-2,1){20}}
\put(246,20){\line(2,1){20}}
\put(26,50){\line(2,1){20}}
\put(266,50){\line(-2,1){20}}
\put(46,60){\line(-2,1){20}}
\put(246,60){\line(2,1){20}}
\put(26,90){\line(2,1){20}}
\put(266,90){\line(-2,1){20}}
\put(46,100){\line(-2,1){20}}
\put(246,100){\line(2,1){20}}
\put(26,60){\oval(6,108)[]}
\put(266,60){\oval(6,108)[]}
\put(26,60){\makebox(0,0)[cc]{$I$}}
\put(266,60){\makebox(0,0)[]{$I$}}
\put(21,60){\oval(30,116)[]}
\put(271,60){\oval(30,116)[]}
\put(15,60){\makebox(0,0)[cc]{$G_0$}}
\put(277,60){\makebox(0,0)[]{$G_0$}}
\end{picture}
\end{center}
\caption{A family of $C_4$-free $4$-edge-connected graphs $G$ of order $n$ and diameter $d$ with $d=\frac{2}{7}n+O(1)$.
The graph $G$ contains two disjoint copies of the graph $G_0$ 
explained after Theorem \ref{lemma2}.}\label{fig1b}
\end{figure}
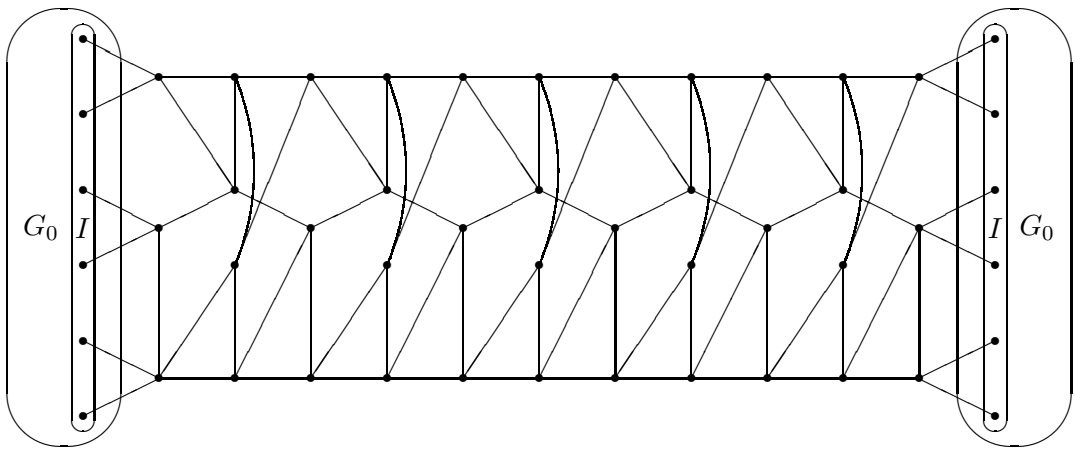
The proofs of (\ref{e1}), (\ref{e2}), and (\ref{e3b}) all rely on rather local arguments
counting vertices within some bounded distance. While this leads to essentially best possible results in some settings, for the proof of (\ref{e3}), 
involving the best possible factor $2/5$, a non-local argument was required.
We believe that also the proof of a best possible version of (\ref{e3b}) 
would require a non-local argument.
The two presented improvements of the consequences of (\ref{e1}) 
for small values of $\lambda$ both require a detailed case analysis,
that is, their proofs do not seem to generalize to larger values of $\lambda$.

The graphs constructed by Erd\H{o}s et al.~\cite{erpapotu} for (\ref{e1b}) 
contain numerous bridges, that is, their edge-connectivity is small.
Analyzing the edge-connectivity of the building blocks $B(q)$ 
for that construction, cf.~Lemma \ref{lemma1} below, 
allows to show the following.

\begin{theorem}\label{lemma2}
For an odd prime power $q$ at least $7$, and every positive integer $k$, 
there is a connected $C_4$-free graph of 
order $n=(q^2+q-1)k+1$, 
diameter $d=4k$, 
and edge-connectivity $\lambda$ at least $q-6$, in particular,
$$d\geq \frac{4(n-1)}{\lambda^2-11\lambda+29}.$$
\end{theorem}
Note that Theorem \ref{lemma2} implies the existence of 
a specific $C_4$-free, $4$-edge-connected graph $G_0$ of bounded order
for which $G_0^2$ contains an independent set $I$ of order $6$,
which is used in Figure \ref{fig1b}.

All proofs are given in the following section.

In a final concluding section, we discuss directions for further research.

\section{Proofs}

We begin with the proof of Theorem \ref{theorem1};
the non-local argument mentioned above 
relies on Claim \ref{claim1}(c)-(g) as well as (\ref{e4}).

\begin{proof}[Proof of Theorem \ref{theorem1}]
Let $G$ be a connected $C_4$-free graph of 
order $n$, 
diameter $d$, 
and edge-connectivity $\lambda$ at least $3$.
Let $u$ be a vertex in $G$ of maximum eccentricity.
For $i\in \{ 0,1,\ldots,d\}$,
let $V_i$ be the set of vertices of $G$ at distance $i$ from $u$.
By the choice of $u$, the number $n_i$ of vertices in $V_i$ is positive for every $i\in \{ 0,1,\ldots,d\}$.

\begin{claim}\
\label{claim1}
\begin{enumerate}[(a)]
\item If $n_i=1$ for some $i\leq d-1$, 
then $n_{i+1}\geq 3$.
\item If $n_i=2$ for some $i\leq d-1$, 
then $n_{i+1}\geq 2$.
\item If $(n_i,n_{i+1})=(1,3)$ for some $i\leq d-1$, 
then $i\leq d-2$ and $n_{i+2}\geq 4$.
\item If $(n_i,n_{i+1})=(2,2)$ for some $i\leq d-1$, 
then $i\leq d-2$ and $n_{i+2}\geq 3$.
\item If $(n_i,n_{i+1})=(2,3)$ for some $i\leq d-1$, 
then $i\leq d-2$ and $n_{i+2}\geq 2$.
\item If $(n_i,n_{i+1},n_{i+2})=(2,3,2)$ for some $i\leq d-2$, 
then $V_{i+1}$ contains exactly one or two edges,
and $i\leq d-3$.
\item If $(n_i,n_{i+1},n_{i+2},n_{i+3})=(2,3,2,2)$ for some $i\leq d-3$, 
then $i\leq d-4$ and $n_{i+4}\geq 4$.
\end{enumerate}
\end{claim}
\begin{proof}[Proof of Claim \ref{claim1}]
(a) and (b) follow because $\lambda\geq 3$ implies $n_in_{i+1}\geq 3$.

\medskip

\noindent (c) Since $\lambda\geq 3$, 
the unique vertex in $V_i$ is adjacent to all three vertices in $V_{i+1}$.
If $i=d-1$, then $V_{i+1}$ must be complete,
which implies the contradiction that $G$ contains some $C_4$.
Hence, we have $i\leq d-2$.
Suppose, for a contradiction, that 
$(n_i,n_{i+1},n_{i+2})=(1,3,p)$ for some $p\leq 3$.
If $p\leq 2$, then some vertex in $V_{i+2}$ has two neighbors in $V_{i+1}$,
which implies the contradiction that $G$ contains some $C_4$.
It follows that $p=3$, and 
that every vertex in $V_{i+2}$ has exactly one neighbor in $V_{i+1}$,
which implies that the edges between $V_{i+1}$ and $V_{i+2}$ form a matching.
Since every vertex in $V_{i+1}$ has degree at least $3$,
the set $V_{i+1}$ contains at least two edges, 
which implies the contradiction that $G$ contains some $C_4$.
This completes the proof of (c).

\medskip

\noindent (d) If $i=d-1$, then $\lambda\geq 3$ implies 
that there are all four possible edges between $V_i$ and $V_{i+1}$,
which implies the contradiction that $G$ contains some $C_4$.
Hence, we have $i\leq d-2$.
Suppose, for a contradiction, that 
$(n_i,n_{i+1},n_{i+2})=(2,2,p)$ for some $p\leq 2$.
Since $\lambda\geq 3$, 
both $V_i$ as well as $V_{i+2}$ contain a vertex 
that is adjacent to both vertices in $V_{i+1}$,
which implies the contradiction that $G$ contains some $C_4$.
This completes the proof of (d).

\medskip

\noindent (e) If $i=d-1$, then the $C_4$-freeness implies 
that two vertices in $V_{i+1}$ have only one neighbor in $V_i$.
Since $\lambda\geq 3$, this implies that $V_{i+1}$ is complete.
Since two vertices in $V_{i+1}$ have a common neighbor in $V_i$,
we obtain the contradiction that $G$ contains some $C_4$.
Hence, we have $i\leq d-2$.
Suppose, for a contradiction, that 
$(n_i,n_{i+1},n_{i+2})=(2,3,1)$.
Since $\lambda\geq 3$, 
the unique vertex in $V_{i+2}$ is adjacent to all three vertices in $V_{i+1}$,
and some vertex in $V_i$ has two neighbors in $V_{i+1}$,
which implies the contradiction that $G$ contains some $C_4$.
This completes the proof of (e).

\medskip

\noindent (f) First, suppose, for a contradiction, 
that $V_{i+1}$ is independent.
Since $\lambda\geq 3$, 
the set $V_{i+1}$ contains
two vertices that are adjacent to both vertices in $V_i$
or $V_{i+1}$ contains
two vertices that are adjacent to both vertices in $V_{i+2}$.
In both cases $G$ contains some $C_4$,
which implies that $V_{i+1}$ is not independent.
Now, suppose, for a contradiction, that $V_{i+1}$ is complete,
then some vertex in $V_{i+2}$ has two neighbors in $V_{i+1}$,
which implies the contradiction that $G$ contains some $C_4$.
Hence, the set $V_{i+1}$ contains exactly one or two edges.
If $i=d-2$, then $\lambda\geq 3$ and the $C_4$-freeness imply 
that both vertices in $V_{i+2}$ are adjacent
and have two neighbors in $V_{i+1}$.
Together with a suitable edge within $V_{i+1}$,
we obtain the contradiction that $G$ contains some $C_4$.
Hence, we have $i\leq d-3$.
This completes the proof of (f).

\medskip

\noindent (g) By (d), 
we obtain $i\leq d-4$ and $n_{i+4}\geq 3$.
Suppose, for a contradiction, that $n_{i+4}=3$.
Let 
$V_{i+1}=\{ a_1,a_2,a_3\}$, 
$V_{i+2}=\{ b_1,b_2\}$,
$V_{i+3}=\{ c_1,c_2\}$, and
$V_{i+4}=\{ d_1,d_2,d_3\}$.

First, we show that 
$V_{i+1}$ contains exactly one edge.
Therefore, by (f), we suppose, for a contradiction, 
that $V_{i+1}$ contains exactly two edges;
say $a_1a_2,a_2a_3\in E(G)$.
Since $\lambda\geq 3$ and $G$ contains no $C_4$,
we may assume, by symmetry,
that $c_1$ is adjacent to $b_1$ and $b_2$,
that $b_1$ is adjacent to $a_1$ and $a_2$,
that $b_2$ is adjacent to $a_3$, and 
that $c_2$ has exactly one neighbor in $V_{i+2}$.
If $c_2$ is adjacent to $b_1$,
then, since $G$ contains no $C_4$,
the vertex $b_2$ has degree $2$, 
which contradicts $\lambda\geq 3$.
Conversely, 
if $c_2$ is adjacent to $b_2$,
then, since $G$ contains no $C_4$,
the two edges $a_3b_2$ and $b_1c_1$ form a $2$-edge cut,
which contradicts $\lambda\geq 3$.
Hence, the set $V_{i+1}$ contains exactly one edge;
say $a_1a_2\in E(G)$.

Next, we show that $V_{i+2}$ is complete.
Therefore, suppose, for a contradiction,
that $b_1$ and $b_2$ are not adjacent.
Since $\lambda\geq 3$ and $G$ contains no $C_4$,
we may assume, by symmetry,
that $c_1$ is adjacent to $b_1$ and $b_2$,
that $a_3$ is adjacent to $b_2$, and
that $b_1$ is adjacent to $a_1$.
Since the two edges $a_3b_2$ and $b_1c_1$ form no $2$-edge cut,
it follows
that $b_2$ is adjacent to $a_2$, and
that $b_1$ is adjacent to $c_2$.
Now, 
the two edges $a_1b_1$ and $b_2c_1$ form a $2$-edge cut,
which contradicts $\lambda\geq 3$.
Hence, the set $V_{i+2}$ is complete.

Since $\lambda\geq 3$ and $G$ contains no $C_4$,
we may assume, by symmetry,
that $c_1$ is adjacent to $b_1$ and $b_2$,
that $c_2$ is adjacent to $b_2$,
that $c_2$ is adjacent to $d_2$ and $d_3$, and
that $c_1$ is adjacent to $d_1$.
Now, 
the two edges $b_2c_2$ and $c_1d_1$ form a $2$-edge cut,
which contradicts $\lambda\geq 3$.
This completes the proof of (g).
\end{proof}
Since $n_0=1$, 
Claim \ref{claim1}(a,c) implies that 
either $n_1\geq 4$
or $n_1=3$ and $n_2\geq 4$. 
Furthermore, if $n_1=4$, then $n_2\geq 3$.
In all cases we obtain 
$\frac{n_0+n_1-\frac{1}{2}}{1+1}\geq \frac{5}{2}$ 
or $\frac{n_0+n_1+n_2-\frac{1}{2}}{2+1}\geq \frac{5}{2}$.
Let $i\in \{ 0,1,\ldots,d\}$ be maximum such that 
\begin{eqnarray}\label{e4}
\frac{n_0+n_1+\cdots+n_i-\frac{1}{2}}{i+1}\geq \frac{5}{2}.
\end{eqnarray}
The above observations imply that $i$ is well-defined and that $i\geq 1$.

If $i=d$, then $\frac{n-\frac{1}{2}}{d+1}\geq \frac{5}{2}$, which implies (\ref{e3}).
Similarly, 
if $i=d-1$, then $\frac{n-1-\frac{1}{2}}{d}\geq \frac{n-n_d-\frac{1}{2}}{d}\geq \frac{5}{2}$, 
which implies (\ref{e3}).
Hence, we may assume that $i\leq d-2$.
The choice of $i$ implies $n_{i+1}\in \{ 1,2\}$.
If $n_{i+1}=1$, then, by Claim \ref{claim1}(a,c),
we obtain $\frac{n_{i+1}+n_{i+2}}{2}\geq \frac{5}{2}$
or $i\leq d-3$ and $\frac{n_{i+1}+n_{i+2}+n_{i+3}}{3}\geq \frac{5}{2}$,
which implies the contradiction 
that $i+2$ or $i+3$ satisfies (\ref{e4}) (replacing $i$ with the larger value).
Hence, we obtain $n_{i+1}=2$.
By the choice of $i$, Claim \ref{claim1}(b) implies $n_{i+2}=2$.
By Claim \ref{claim1}(d), we obtain $i\leq d-3$ and $n_{i+3}\geq 3$.
By the choice of $i$, we have $n_{i+3}=3$.
Let the positive integer $k$ be maximum such that $i+2k+1\leq d$ and
$$(n_{i+1},n_{i+2},n_{i+3},\ldots,n_{i+2k},n_{i+2k+1})
=(2,\underbrace{2,3,\ldots,2,3}_{k \times (2,3)}).$$
By Claim \ref{claim1}(d) and the choice of $i$,
we obtain $i+2k+2\leq d$ and $n_{i+2k+2}=2$.
By Claim \ref{claim1}(b)(f) and the choices of $i$ and $k$, 
we obtain $i+2k+3\leq d$ and $n_{i+2k+3}=2$.
By Claim \ref{claim1}(g), 
we obtain $i+2k+4\leq d$ and $n_{i+2k+4}\geq 4$.
Now, we have
$\frac{n_{i+1}+\cdots+n_{i+2k+4}}{2k+4}\geq \frac{5}{2}$,
which implies the contradiction 
that $i+2k+4$ satisfies (\ref{e4}) (replacing $i$ with this larger value).
This final contradiction completes the proof.
\end{proof}

The following proof of Theorem \ref{theorem1b}
is similar to the general approach from \cite{fumamu}.

\begin{proof}[Proof of Theorem \ref{theorem1b}]
Let $G$ be as in the statement.
Let $u$, $V_i$, and $n_i$ be as in the proof of Theorem \ref{theorem1}.
\begin{claim}\label{claim1b}
\begin{eqnarray}\label{e3c}
n_{i-1} + n_i + n_{i+1} & \geq & 9\mbox{ for every }i \in \{ 1,\ldots ,d-1 \}.
\end{eqnarray}
\end{claim}
\begin{proof}[Proof of Claim \ref{claim1b}]
Let $i \in \{ 1,\ldots ,d-1 \}$.
We consider different cases according to the value of $n_i$.

First, we assume that $n_i=1$.
Since $\lambda\geq 4$, we have $n_{i-1},n_{i+1}\geq 4$, which implies (\ref{e3c}).

Next, we assume that $n_i=2$.
This implies $n_{i-1},n_{i+1}\geq 3$, since, otherwise,
the at least $4$ edges between $V_i$ and either $V_{i-1}$ or $V_{i+1}$ already yield a $C_4$.
If $n_{i-1}=n_{i+1}=3$,
then some vertex in $V_{i-1}$ as well as some vertex in $V_{i+1}$ is adjacent to both vertices in $V_i$,
which yields a $C_4$.
Hence, $\min\{ n_{i-1},n_{i+1}\}\geq 3$ and $\max\{ n_{i-1},n_{i+1}\}\geq 4$,
which implies (\ref{e3c}).

Next, we assume that $n_i=3$.
Let $V_i=\{ b_1,b_2,b_3\}$.
Since $\lambda\geq 4$, we obtain $n_{i-1},n_{i+1}\geq 2$.
If $(n_{i-1},n_{i+1})=(2,3)$, then some vertex $c_1$ in $V_{i+1}$ has two neighbors, say $b_1$ and $b_3$, in $V_i$.
Since $\lambda\geq 4$ and $G$ is $C_4$-free, 
one vertex in $V_{i-1}$ is adjacent to $b_1$ and $b_2$,
and 
the other vertex in $V_{i-1}$ is adjacent to $b_2$ and $b_3$.
Since $G$ is $C_4$-free, there is at most one edge with both endpoints in $V_i$.
Since $\delta\geq\lambda\geq 4$, this implies the existence of at least four edges between $V_i$ and $V_{i+1}$
that are distinct from the two edges $b_1c_1$ and $b_3c_1$,
which easily implies the contradiction that $G$ contains a $C_4$
using one vertex in $V_{i-1}$, two vertices in $V_i$, and one vertex in $V_{i+1}$.
Similarly, the assumption $(n_{i-1},n_{i+1})\in \{ (2,2),(3,2)\}$ yields a contradiction.
Hence, we obtain $n_{i-1}+n_{i+1}\geq 6$,
which implies (\ref{e3c}).

Next, we assume that $n_i=4$.
If $n_{i-1}=1$, then no vertex in $V_{i+1}$ has two neighbors in $V_i$, which implies $n_{i+1}\geq 4$,
and, hence, (\ref{e3c}).
Similarly, if $n_{i+1}=1$, then (\ref{e3c}) follows.
Suppose, for a contradiction, that $(n_{i-1},n_{i+1})=(2,2)$.
Let $V_{i-1}=\{ a_1,a_2\}$, 
$V_i=\{ b_1,b_2,b_3,b_4\}$, and
$V_{i+1}=\{ c_1,c_2\}$.
Since $\lambda\geq 4$ and $G$ is $C_4$-free, 
we may assume that 
$c_1$ is adjacent to $b_1$ and $b_2$,
$b_1$ is adjacent to $a_1$, and
$b_2$ is adjacent to $a_2$.
Since every vertex in $V_i$ has a neighbor in $V_{i-1}$,
this implies that $c_1$ has no further neighbor in $V_i$.
Hence, the vertex $c_2$ has exactly two neighbors in $V_i$.
First, suppose that $c_2$ is adjacent to $b_3$ and $b_4$.
By symmetry, we may assume that 
$b_3$ is adjacent to $a_1$, and
$b_4$ is adjacent to $a_2$.
Now, $b_2$ has two neighbors in $V_i$.
Since $G$ is $C_4$-free, this implies that $b_2$ is adjacent to $b_1$ and $b_4$.
Symmetrically, it follows that $b_3$ is adjacent to $b_4$ and $b_1$,
which implies the contradiction that $G$ contains a $C_4$ completely within $V_i$.
Next, suppose that $c_2$ is adjacent to $b_2$ and $b_3$.
Since $\lambda\geq 4$ and $G$ is $C_4$-free, 
the vertex $b_4$ has at most one neighbor in $V_i$, and at most one neighbor in $V_{i+1}$.
Hence, $b_4$ is adjacent to $a_1$ and $a_2$.
Since $G$ is $C_4$-free, $b_4$ has no neighbor in $V_{i+1}$,
which implies the contradiction that $b_4$ has degree at most $3$.
Altogether, we obtain $n_{i-1}+n_{i+1}\geq 5$, 
which implies (\ref{e3c}).

Next, we assume that $n_i\in \{ 5,6\}$.
If $(n_{i-1},n_{i+1})$ is $(1,1)$, $(1,2)$, or $(2,1)$, then some vertex in $V_{i+1}$ has two neighbors in $V_i$
that have a common neighbor in $V_{i-1}$, which is a contradiction.
This implies (\ref{e3c}) in this case.

Finally, since $n_{i-1}+n_{i+1}\geq 2$, 
the inequality (\ref{e3c}) is trivial for $n_i\geq 7$,
which completes the proof of the claim.
\end{proof}
Since $\lambda \geq 4$ and $G$ is $C_4$-free, 
we have $n_{d-1}+ n_d\geq 5$.
Now, together with Claim \ref{claim1b} this implies that
\begin{itemize}
\item 
if $d\equiv 2$ mod $3$,
then $n=\underbrace{(n_0+n_1+n_2)}_{\geq 9}
+\underbrace{(n_3+n_4+n_5)}_{\geq 9}+\cdots+
\underbrace{(n_{d-2}+n_{d-1}+n_d)}_{\geq 9}
\geq 9(d+1)/3$, 
\item if $d\equiv 0$ mod $3$,
then $n=(n_0+n_1+n_2)+\cdots+(n_{d-3}+n_{d-2}+n_{d-1})+\underbrace{n_d}_{\geq 1}
\geq 9d/3+1$, and
\item if $d\equiv 1$ mod $3$,
then $n=(n_0+n_1+n_2)+\cdots+(n_{d-4}+n_{d-3}+n_{d-2})+\underbrace{(n_{d-1}+n_d)}_{\geq 5}
\geq 9(d-1)/3+5$.
\end{itemize}
In all three cases, we obtain (\ref{e3b}).
\end{proof}

For the proof of Theorem \ref{lemma2},
we need to explain the construction behind (\ref{e1b}).

Let $q$ be an odd prime power, and let $\mathbb{F}_q$ denote the field of order $q$.
Two non-zero vectors from $\mathbb{F}_q^3$
are considered {\it equivalent} if they generate the same one-dimensional subspace of $\mathbb{F}_q^3$, that is, they are non-zero multiples of each other.
The {\it Brown graph} $B(q)$ has as its vertices the equivalence classes $[x]$ 
of the non-zero vectors $x$ from $\mathbb{F}_q^3$, 
and two distinct vertices $[x]$ and $[y]$ are adjacent exactly if $xy^T=0$.
This graph 
was proposed independently by Brown \cite{br} 
and Erd\H{o}s and R\'{e}nyi \cite{erre}.
It is a dense $C_4$-free graph with the following properties \cite{basi}:
\begin{itemize}
\item The vertex set $V(B(q))$ of $B(q)$ can be partitioned into two sets $W$ and $V$,
where 
$$W=\{ [x]:x\in \mathbb{F}_q^3\setminus \{ 0\}\mbox{ with }xx^T=0\}$$
contains the so-called {\it quadric} vertices.
The vertices in $W$ have degree $q$ and the vertices in $V$ have degree $q+1$.
Furthermore, 
the order of $B(q)$ is $q^2+q+1$,
$|W|=q+1$, and, hence, $|V|=q^2$.
\item $W$ is an independent set, and 
every vertex from $V$ has exactly two or zero neighbors in $W$.
\item No vertex from $W$ lies on a triangle.
\item Every two non-adjacent vertices 
as well as every two vertices from $V$ (adjacent or not) 
have exactly one common neighbor.
\end{itemize}
For the construction of the graph $G$ in (\ref{e1b}), 
Erd\H{o}s et al.~\cite{erpapotu} modify $B(q)$ as follows:
Let $c$ be a quadric vertex of $B(q)$, and let $a$ and $b$ be two neighbors of $c$.
Since $c$ is quadric, 
the vertices $a$ and $b$ are not adjacent and both of degree $q+1$.

Let 
\begin{eqnarray*}
A&=&N_{B(q)}(a)\setminus N_{B(q)}(b)=\{ a_0,a_1,\ldots,a_{q-1}\}\mbox{ and }\\
B&=&N_{B(q)}(b)\setminus N_{B(q)}(a)=\{ b_0,b_1,\ldots,b_{q-1}\}
\end{eqnarray*}
be such that $a_0$ and $b_0$ are the second quadric neighbors of $a$ and $b$ distinct from $c$, respectively.
Since 
no quadric vertex lies on a triangle,
every vertex in $A$ has exactly one common neighbor with $b$, and 
every vertex in $B$ has exactly one common neighbor with $a$,
the edges of $B(q)$ between $A$ and $B$ form a perfect matching $M$ between these two sets.
Possibly by renaming vertices, we may assume that 
$$M=\{ a_0b_1,a_1b_0\}\cup \{ a_2b_2,a_3b_3,\ldots,a_{q-1}b_{q-1}\}.$$
For every $i\in [q-1]$, the two non-quadric vertices $a_i$ and $b_i$ have exactly one common neighbor $c_i$.
By the properties of $B(q)$, the vertices $c_i$ are all distinct, and do not belong to $\{ c\}\cup A\cup B$.

Let 
$$C=\{ c_1,\ldots,c_{q-1}\}.$$
Let the graph $H$ arise from $B(q)$ by removing the quadric vertex $c$ as well as all edges from $M$,
cf.~Figure \ref{fig2}.

\begin{figure}[H]
\begin{center}
\unitlength 1mm 
\linethickness{0.4pt}
\ifx\plotpoint\undefined\newsavebox{\plotpoint}\fi 
\begin{picture}(46,84)(0,0)
\put(8,10){\circle*{1}}
\put(8,30){\circle*{1}}
\put(8,40){\circle*{1}}
\put(8,50){\circle*{1}}
\put(8,60){\circle*{1}}
\put(8,80){\circle*{1}}
\put(38,10){\circle*{1}}
\put(38,30){\circle*{1}}
\put(38,40){\circle*{1}}
\put(38,50){\circle*{1}}
\put(38,60){\circle*{1}}
\put(38,80){\circle*{1}}
\put(23,5){\circle*{1}}
\put(23,25){\circle*{1}}
\put(23,35){\circle*{1}}
\put(23,45){\circle*{1}}
\put(8,10){\line(1,0){30}}
\put(8,30){\line(1,0){30}}
\put(8,40){\line(1,0){30}}
\put(38,10){\line(-3,-1){15}}
\put(38,30){\line(-3,-1){15}}
\put(38,40){\line(-3,-1){15}}
\put(38,50){\line(-3,-1){15}}
\put(23,5){\line(-3,1){15}}
\put(23,25){\line(-3,1){15}}
\put(23,35){\line(-3,1){15}}
\put(23,45){\line(-3,1){15}}
\put(4,10){\makebox(0,0)[cc]{\footnotesize $a_{q-1}$}}
\put(4,30){\makebox(0,0)[cc]{$a_3$}}
\put(4,40){\makebox(0,0)[cc]{$a_2$}}
\put(4,50){\makebox(0,0)[cc]{$a_1$}}
\put(4,60){\makebox(0,0)[cc]{$a_0$}}
\put(4,80){\makebox(0,0)[cc]{$a$}}
\put(42,10){\makebox(0,0)[cc]{\footnotesize $b_{q-1}$}}
\put(42,30){\makebox(0,0)[cc]{$b_3$}}
\put(42,40){\makebox(0,0)[cc]{$b_2$}}
\put(42,50){\makebox(0,0)[cc]{$b_1$}}
\put(42,60){\makebox(0,0)[cc]{$b_0$}}
\put(42,80){\makebox(0,0)[cc]{$b$}}
\put(27.2,3.8){\makebox(0,0)[cc]{\footnotesize $c_{q-1}$}}
\put(27,24){\makebox(0,0)[cc]{$c_3$}}
\put(27,34){\makebox(0,0)[cc]{$c_2$}}
\put(27,44){\makebox(0,0)[cc]{$c_1$}}
\put(38,20){\makebox(0,0)[cc]{$\vdots$}}
\put(8,20){\makebox(0,0)[cc]{$\vdots$}}
\put(23,15){\makebox(0,0)[cc]{$\vdots$}}
\put(8,60){\line(3,-1){30}}
\put(38,60){\line(-3,-1){30}}
\put(23,80){\circle*{1}}
\put(23,84){\makebox(0,0)[cc]{$c$}}
\put(23,80){\line(1,0){15}}
\put(23,80){\line(-1,0){15}}
\put(0,7){\framebox(11,56)[cc]{}}
\put(35,7){\framebox(11,56)[cc]{}}
\put(16,2){\framebox(14,45)[cc]{}}
\put(5,3){\makebox(0,0)[cc]{$A$}}
\put(41,3){\makebox(0,0)[cc]{$B$}}
\put(23,-2){\makebox(0,0)[cc]{$C$}}
\put(8,80){\line(0,-1){17}}
\put(38,80){\line(0,-1){17}}
\end{picture}
\end{center}
\caption{The structure of $B(q)$ around $\{ a,b,c\}\cup A\cup B\cup C$; 
the figure shows all edges except for possible edges within each of the sets $A$, $B$, and $C$.}\label{fig2}
\end{figure}
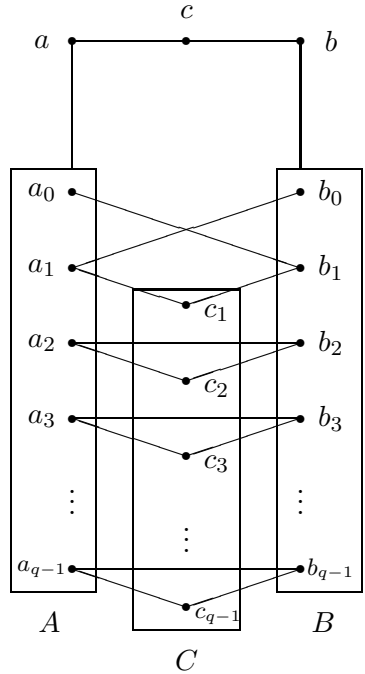
The idea behind the construction of $H$ is to destroy all paths in $B(q)$ 
of length at most two between $a$ and $b$
increasing their distance to four.
Now, the graph $G$ considered for (\ref{e1b}) 
arises from the disjoint union of copies $H_1,\ldots,H_k$ of $H$
by adding edges between the vertex $b$ from $H_i$
and the vertex $a$ from $H_{i+1}$ for every $i\in [k-1]$.
The graph $G$ has minimum degree $q-1$,
diameter $5k-1$, and 
order $(q^2+q)k$,
which implies (\ref{e1b}).
Note that $G$ contains bridges for $k\geq 2$, that is, 
while its minimum degree is large, it is not even $2$-edge-connected
regardless of the value of $q$.

For the proof of Theorem \ref{lemma2},
we consider the edge-connectivity of $H$.

\begin{lemma}\label{lemma1}
The graph $H$ is $(q-6)$-vertex-connected.
\end{lemma}
\begin{proof}
We show the existence of at least $q-6$ internally vertex-disjoint paths 
between any two distinct vertices $x$ and $y$ of $H$. 
In $H$, the two vertices $a$ and $b$ are connected 
by the $q-1$ internally vertex-disjoint paths
$aa_1c_1b_1b,\ldots,aa_{q-1}c_{q-1}b_{q-1}b$, cf.~Figure \ref{fig2}. 

Now, let $\{ x,y\}\not=\{ a,b\}$. 
We consider different cases.

First, we assume that $x$ and $y$ are not adjacent.
By symmetry, we may assume that $d_{B(q)}(x)\leq d_{B(q)}(y)$, and that $x\not\in \{ a,b\}$.
By the properties of $B(q)$,
the set $N_{B(q)}(x)\setminus N_{B(q)}(y)$ 
contains a set $X$ of $q-1$ vertices $x_1,\ldots,x_{q-1}$,
and each vertex $x_i$ from $X$ has exactly one common neighbor $y_i$ with $y$.
Let $Y=\{ y_1,\ldots,y_{q-1}\}$.
By the properties of $B(q)$, the set $Y$ is disjoint from $\{ x,y\}\cup X$, 
and the $q-1$ elements of $Y$ are all distinct.
In $B(q)$, there are the $q-1$ internally vertex-disjoint paths
$$xx_1y_1y,\ldots,xx_{q-1}y_{q-1}y$$ 
between $x$ and $y$.
At most one of these paths contains the vertex $c$.
If four more of these paths contain an edge from $M$,
then three vertices from $X$ are incident with edges from $M$.
Since the edges in $M$ connect neighbors of $a$ with neighbors of $b$,
this implies that at least two vertices from $X$ are neighbors of either $a$ or $b$.
Now, the vertex $x\not\in \{ a,b\}$ shares two neighbors with $a$ or $b$, 
which is a contradiction to the properties of $B(q)$.
Hence, at most three of the above $q-1$ paths contain an edge from $M$, and, 
in $H$, there are at least $q-1-1-3$ internally vertex-disjoint paths between $x$ and $y$.

For the rest of the proof, we may assume that $x$ and $y$ are adjacent.

Next, we assume that $x,y\not\in \{ a,b\}$.
By symmetry, we may assume that $d_{B(q)}(x)\leq d_{B(q)}(y)$,
in particular, the vertex $y$ is not quadric.
By the properties of $B(q)$, regardless of whether $x$ is quadric or not,
\begin{itemize}
\item the set $N_{B(q)}(x)\setminus N_{B(q)}[y]$ 
contains a set $X$ of $q-1$ vertices $x_1,\ldots,x_{q-1}$,
\item the set $N_{B(q)}(y)\setminus N_{B(q)}[x]$ 
contains a set $Y$ of $q-1$ vertices $y_1,\ldots,y_{q-1}$, 
\item there is no edge between $X$ and $Y$, and,
\item for every $i\in [q-1]$, the vertex $x_i$ and the vertex $y_i$ have a unique common neighbor $z_i$.
\end{itemize}
Let $Z=\{ z_1,\ldots,z_{q-1}\}$.
By the properties of $B(q)$,
the $q-1$ elements of $Z$ are all distinct, and 
$Z$ is disjoint from $\{ x,y\}\cup X\cup Y$.
In $B(q)$, there are the $q-1$ internally vertex-disjoint paths
$$xx_1z_1y_1y,\ldots,xx_{q-1}z_{q-1}y_{q-1}y$$ 
between $x$ and $y$.
At most one of these paths contains the vertex $c$.
If five of these paths contain an edge from $M$,
then three vertices from either $X$ or $Y$ are incident edges from $M$.
Similarly as before, this implies that one of the two vertices $x$ and $y$
has two common neighbors with one of the two vertices $a$ and $b$,
which is a contradiction to the properties of $B(q)$.
Hence, at most four of the above $q-1$ paths contain an edge from $M$, and, 
in $H$, there are at least $q-1-1-4$ internally vertex-disjoint paths between $x$ and $y$.

For the rest of the proof, we may assume, by the symmetry between $a$ and $b$, that $y=a$.

As $x$ and $y$ are adjacent, this implies $x\in A$.

First, we assume that $x=a_0$, that is, the vertex $a_0$ 
is the unique second quadric neighbor of $a$ distinct from $c$.
The vertex $x=a_0$ has no neighbor in $\{ b,c\}\cup (B\setminus \{ b_1\})\cup C$.
Let $X=N_{B(q)}(x)\setminus \{ a,b_1\}=\{ x_2,\ldots,x_{q-1}\}$
and $Y=\{ a_2,\ldots,a_{q-1}\}\subseteq A$.
For every $i\in \{ 2,\ldots,q-1\}$,
the vertices $x_i$ and $a_i$ have a unique common neighbor $z_i$.
By the properties of $B(q)$, the vertices $z_2,\ldots,z_{q-1}$ are all distinct.
Note that $z_i$ may coincide with $b_i$ or $c_i$ 
but is distinct from $c$ and $c_j$ for $j\not=i$,
cf.~Figure \ref{fig3} for an illustration.
\begin{figure}[H]
\begin{center}
\unitlength 1mm 
\linethickness{0.4pt}
\ifx\plotpoint\undefined\newsavebox{\plotpoint}\fi 
\begin{picture}(71,84)(0,0)
\put(33,10){\circle*{1}}
\put(8,10){\circle*{1}}
\put(33,30){\circle*{1}}
\put(8,30){\circle*{1}}
\put(33,40){\circle*{1}}
\put(8,40){\circle*{1}}
\put(33,50){\circle*{1}}
\put(33,60){\circle*{1}}
\put(33,80){\circle*{1}}
\put(63,10){\circle*{1}}
\put(63,30){\circle*{1}}
\put(63,40){\circle*{1}}
\put(63,50){\circle*{1}}
\put(63,60){\circle*{1}}
\put(63,80){\circle*{1}}
\put(48,5){\circle*{1}}
\put(48,25){\circle*{1}}
\put(48,35){\circle*{1}}
\put(48,45){\circle*{1}}
\put(33,10){\line(1,0){30}}
\put(33,30){\line(1,0){30}}
\put(33,40){\line(1,0){30}}
\put(63,10){\line(-3,-1){15}}
\put(63,30){\line(-3,-1){15}}
\put(63,40){\line(-3,-1){15}}
\put(63,50){\line(-3,-1){15}}
\put(48,5){\line(-3,1){15}}
\put(48,25){\line(-3,1){15}}
\put(48,35){\line(-3,1){15}}
\put(48,45){\line(-3,1){15}}
\put(29,80){\makebox(0,0)[cc]{$a$}}
\put(67,80){\makebox(0,0)[cc]{$b$}}
\put(63,20){\makebox(0,0)[cc]{$\vdots$}}
\put(33,20){\makebox(0,0)[cc]{$\vdots$}}
\put(8,20){\makebox(0,0)[cc]{$\vdots$}}
\put(48,15){\makebox(0,0)[cc]{$\vdots$}}
\put(33,60){\line(3,-1){30}}
\put(63,60){\line(-3,-1){30}}
\put(48,80){\circle*{1}}
\put(48,84){\makebox(0,0)[cc]{$c$}}
\put(48,80){\line(1,0){15}}
\put(48,80){\line(-1,0){15}}
\put(25,7){\framebox(11,56)[cc]{}}
\put(60,7){\framebox(11,56)[cc]{}}
\put(41,2){\framebox(14,45)[cc]{}}
\put(30,3){\makebox(0,0)[cc]{$A$}}
\put(66,3){\makebox(0,0)[cc]{$B$}}
\put(48,-2){\makebox(0,0)[cc]{$C$}}
\put(33,80){\line(0,-1){17}}
\put(63,80){\line(0,-1){17}}
\put(1,7){\framebox(14,36)[cc]{}}
\put(8,3){\makebox(0,0)[cc]{$X$}}
\put(20,40){\circle*{1}}
\put(8,40){\line(1,0){12}}
\put(20,40){\line(1,0){13}}
\qbezier(8,30)(58.5,12)(63,30)
\qbezier(8,10)(26,-7.5)(48,5)
\put(20,37){\makebox(0,0)[cc]{$z_2$}}
\put(8,37){\makebox(0,0)[cc]{$x_2$}}
\put(33,37){\makebox(0,0)[cc]{$a_2$}}
\put(8,27){\makebox(0,0)[cc]{$x_3$}}
\put(67,31){\makebox(0,0)[cc]{$z_3$}}
\put(52,4){\makebox(0,0)[cc]{\footnotesize $z_{q-1}$}}
\put(8,13){\makebox(0,0)[cc]{$x_{q-1}$}}
\end{picture}
\end{center}
\caption{Local configuration for the case $(x,y)=(a_0,a)$ 
illustrating different options for the common neighbor $z_i$ of $x_i$ and $a_i$
for $i\in \{ 2,3,q-1\}$.}\label{fig3}
\end{figure}
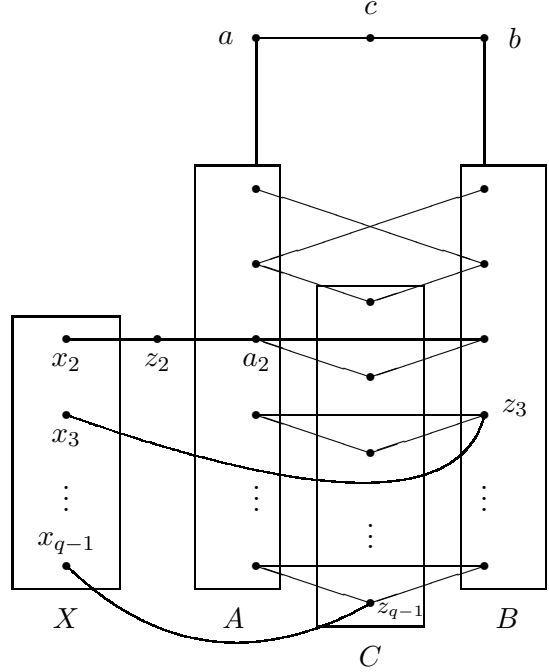
If $z_i=b_i$, then 
let $P_i$ be the path $xx_ib_ic_ia_iy$ avoiding the edge $a_ib_i$ from $M$,
otherwise, 
let $P_i$ be the path $xx_iz_ia_iy$.
In $H$, there are the $q-2$ internally vertex-disjoint paths
$P_2,\ldots,P_{q-1}$ between $x$ and $y$.

Finally, we assume that $x=a_i$ for some $i\in [q-1]$.
The vertex $x=a_i$ has 
exactly one neighbor $b_j$ in $B$,
exactly one neighbor $c_i$ in $C$, 
exactly one neighbor in $A$ --- its unique common neighbor with $a$, and 
no neighbor in $\{ b,c\}$.
Hence, the set 
$X=N_{B(q)}(x)\setminus (\{ a,b_j,c_i\}\cup A)$
contains $q-3$ vertices $x_3,\ldots,x_{q-1}$.
Let $Y=\{ a_3,\ldots,a_{q-1}\}\subseteq A$.
For every $i\in \{ 3,\ldots,q-1\}$,
the vertices $x_i$ and $a_i$ have a unique common neighbor $z_i$.
By the properties of $B(q)$, the vertices $z_3,\ldots,z_{q-1}$ are all distinct.
Again, the vertex $z_i$ may coincide with $b_i$ or $c_i$ 
but is distinct from $c$ and $c_j$ for $j\not=i$.
Defining the paths $P_3,\ldots,P_{q-1}$ as above,
we obtain that, in $H$, 
there are the $q-3$ internally vertex-disjoint paths between $x$ and $y$.

This completes the proof.
\end{proof}
Adapting the proof of Lemma \ref{lemma1},
it is not difficult to show that $B(q)$ is actually $q$-vertex-connected.

\begin{proof}[Proof of Theorem \ref{lemma2}]
The corresponding graph arises from the disjoint union of copies $H_1,\ldots,H_k$ of $H$
by identifying the vertex $b$ from $H_i$
with the vertex $a$ from $H_{i+1}$ for every $i\in [k-1]$.
\end{proof}

\section{Conclusion}

Our results motivate several research problems.
In Theorem \ref{theorem1}, 
one could determine the best possible value for the additive constant.
In fact, we believe that the graphs in Figure \ref{fig1} are extremal.
In Theorem \ref{theorem1b},
one could determine the best possible factor.
While, unfortunately, this seems to require a detailed and tedious case analysis
in combination with a non-local argument,
we believe that our approach from Theorem \ref{theorem1} could be adapted.
Finally, and most interestingly, 
one should determine the smallest possible constant $c$ such that 
$d\leq \frac{cn}{\lambda^2+O(\lambda)}$
for every $C_4$-free graph of order $n$, diameter $d$,
and edge-connectivity $\lambda$.
Note that (\ref{e1}) and Theorem \ref{lemma2} imply $4\leq c\leq 5$.

\end{document}